\documentclass[reqno,11pt,a4paper]{amsart}

\usepackage{amssymb}
\usepackage{url}
\usepackage{hyperref}
\usepackage{pdfpages}
\usepackage{graphicx}
\allowdisplaybreaks[4]

\usepackage{mathrsfs}
\let\mathcal\mathscr
\usepackage[mathscr]{eucal}
\usepackage[all]{xy}
\SelectTips{cm}{}

\usepackage{color}

\let\phi=\varphi
\let\kappa=\varkappa
\newcommand*{\ldb}{[\![}
\newcommand*{\rdb}{]\!]}

\DeclareMathOperator{\CDiff}{\mathcal{C}Diff}

\newcommand*{\eval}[1]{\left.#1\right|}
\newcommand*{\abs}[1]{\left|#1\right|}
\newcommand*{\Ev}{\mathbf{E}}

\theoremstyle{theorem}
\newtheorem{proposition}{Proposition}

\newtheorem{theorem}{Theorem}

\theoremstyle{definition}

\newtheorem{definition}{Definition}

\theoremstyle{remark}
\newtheorem{remark}{Remark}

\usepackage[all]{xy}
\SelectTips{cm}{}

\usepackage{mathrsfs}
\let\mathcal\mathscr
\usepackage[mathscr]{eucal}

\newcommand{\cprime}{\/{\mathsurround=0pt$'$}}

\numberwithin{equation}{section}

\begin{document}

\title{On variational bivectors}
\author{I.S.~Krasil'shchik}
\address{Trapeznikov Institute of Control
    Sciences, 65 Profsoyuznaya street, Moscow 117997,
    Russia}\email{josephkra@gmail.com}
\thanks{Partially supported by the RSF Grant 25-71-20008}
\begin{abstract}
  We construct examples of variational bivectors that are not Poissonian.
\end{abstract}
\date{\today}
\maketitle

\section{Introduction}
\label{sec:introduction}

It was proved in~\cite[Theorem~1]{integrability-conditions} that for
evolutionary equations $u_t = f$ with the non-degenerate symbol of the
right-hand side all variational multi-vectors of degree~$\ge 3$ are
trivial. In particular, it follows from this fact that any variational
bivector defines a Hamiltonian operator (a Poisson structure) for such an
equation (all the necessary definitions will be given in
Section~\ref{sec:vari-bivect}).

We show below that for the equations $u_{xy} = 0$, $u_{xx} + u_{yy}$, and
$u_{xyz} = 0$ (which can obviously be transformed to the evolutionary form,
but with the degenerate symbol) there do exist variational bivectors whose
Schouten bracket does not vanish. So, they are not Hamiltonian structures and
their brackets are nontrivial $3$-vectors.

\section{Variational bivectors}
\label{sec:vari-bivect}

In this section, we follow~\cite{GeomJetSp, KVV-Springer}.

Let $\mathcal{E}\subset J^\infty(\pi)$ be an infinitely prolonged equation,
where $\pi\colon E\to M$ is a vector bundle, see~\cite{Bocharov-et-al}. For
the sake of simplicity, we assume $M = \mathbb{R}^n$ and
$E = \mathbb{R}^m\times\mathbb{R}^n$ with the coordinates $x^1,\dots,x^n$ and
$u^1,\dots,u^m$ in~$ \mathbb{R}^n$ and~$ \mathbb{R}^m$, respectively. Assume
also that~$\mathcal{E}$ is defined by the relation $F = 0$, where
$F = (F^1,\dots,F^m)$ is a vector-function on~$J^\infty(\pi)$. Denote by~$P$
the module of such functions. Denote also by~$\kappa$ the module of sections
of the pull-back~$\pi_\infty^*(\pi)$, where $\pi_\infty\colon J^\infty(\pi)\to
M$. Let~$\mathcal{F}$ be the algebra of smooth functions on~$J^\infty(\pi)$
(or on~$\mathcal{E}$). We use the
notation~$\hat{\bullet} = \hom_{\mathcal{F}}(\bullet, \Lambda_h^n)$, where
$\Lambda_h^n$ is the module of horizontal $n$-forms. We also denote
by~$\CDiff(A, B)$ the module of all $\mathcal{C}$-differential operators
(i.e., operators in total derivatives) $A\to B$, $A$ ans~$B$ being
$\mathcal{F}$-modules.  Let~$\ell_F$ denote the linearization of~$F\in P$ and
$\ell_{\mathcal{E}} = \eval{\ell_F}_{\mathcal{E}}\colon\kappa\to P$.

We impose some regularity conditions on~$\mathcal{E}$:
\begin{enumerate}
\item $\mathcal{E}$ is a smooth submanifold in~$J^\infty(\pi)$;
\item $\mathcal{E}$ surjectively projects by~$\pi$ onto~$J^0(\pi)$;
\item if an ``object'' (function, differential form,
  $\mathcal{C}$-differential operator, etc.) $O\in\mathcal{O}$ admits a
  restriction~$\eval{O}_{\mathcal{E}}$ from~$J^\infty(\pi)$ to~$\mathcal{E}$
  and the restriction vanishes, then
  \begin{equation}
    \label{eq:1}
    O = \Delta(F), \qquad \Delta\in\CDiff(P, \mathcal{O});
  \end{equation}
\item the compatibility complex of~$\ell_{\mathcal{E}}$ is trivial.
\end{enumerate}

\begin{definition}
  Let~$\mathcal{E} = \{F = 0\}\subset J^\infty(\pi)$. Then
  $\mathcal{T}^*\mathcal{E} = \{F = 0, \ell_{\mathcal{E}}(p) = 0\}$ is called
  the cotangent equation of~$\mathcal{E}$, while
  $\mathrm{t}^*\colon \mathcal{T}^*\mathcal{E} \to \mathcal{E}$ is the
  cotangent covering. Here~$p = (p^1,\dots,p^m)$ is an odd variable of
  parity~$1$.
\end{definition}

\begin{definition}
  A $\mathcal{C}$-differential operator $H\colon \hat{P}\to\kappa$ is a
  (variational) bivector on~$\mathcal{E}$ if
  \begin{gather}
    \label{eq:2}
    \ell_{\mathcal{E}}(H) = 0,\\
    \label{eq:3}
    (\ell_{\mathcal{E}}\circ H)^* \equiv H^*\circ\ell_{\mathcal{E}}^* =
    \ell_{\mathcal{E}}\circ H.
  \end{gather} 
\end{definition}

\begin{remark}
  To any $H = (\sum_\sigma h_{ij}^\sigma D_\sigma)\in\CDiff(\hat{P}, \kappa)$,
  where $D_\sigma$ denotes the composition of total derivatives corresponding
  to the multi-index~$\sigma$, the vector-function
  $H_u = (\sum_{\sigma,j} h_{ij}^\sigma p_\sigma^j)$ on~$\mathcal{T}^*\mathcal{E}$
  corresponds. The expression~$\ell_{\mathcal{E}}(H)$ in Equation~\eqref{eq:2}
  is to be understood as~$\ell_{\mathcal{E}}(H_u)$.
\end{remark}

\begin{theorem}[cf.\ \cite{GeomJetSp, KVV-Springer}]
  Let $H$ be a bivector on~$\mathcal{E}$. Then there exists an element~$H_p$
  of parity~$2$ such that the evolutionary derivation~$\Ev_{\phi(H)}$ of
  parity~$1$ with the generating section~$\phi(H) = (H_u, H_p)$ is a symmetry
  of~$\mathcal{T}^*\mathcal{E}$.
\end{theorem}

\begin{proof}
  Using Equations~\eqref{eq:1} and~\eqref{eq:3}, we see that
  \begin{equation}\label{eq:8}
    \ell_F(H(p)) - H^*(\ell_F(p)) = \nabla(F, p),
  \end{equation}
  on the ambient~$J^\infty(\hat{\pi})$, where
  $\Delta\colon \kappa\times\hat{P}\to P$ is a bi-differential operator in
  total derivatives. Then
  \begin{equation*}\label{eq:7}
    H_p = \eval{-\frac{1}{2}\nabla^{*_1}(p, p)}_{\mathcal{E}}
  \end{equation*}
  is the desired element. Here~$\Delta^{*_1}$ denotes the adjoint to~$\Delta$
  with respect to the first argument.
\end{proof}

\begin{remark}
  Due to Equation~\eqref{eq:2}, the operator~$H$ takes symmetries
  of~$\mathcal{E}$ to its cosymmetries. Then, for two conservation
  laws~$\omega$ and~$\omega'$, one can define their
  bracket~$\{\omega, \omega'\}_H$ in the standard way:
  \begin{equation*}
    \mathcal{L}_{\delta(\omega)}(\delta(\omega')) = \delta(\{\omega,
    \omega'\}_H), 
  \end{equation*}
  where~$\mathcal{L}$ is the Lie derivative and $\delta\colon E_1^{0,n-1}\to
  E_1^{1,n-1}$ is the differential in the first term of Vinogradov's spectral
  sequence,~\cite{Vin-C-spec-DAN}.
\end{remark}

\begin{definition}
  For the two bivectors~$H$ and~$H'$, define their Schouten
  bracket~$\ldb H, H'\rdb$ by
  \begin{equation*}
    [\Ev_{\phi(H)}, \Ev_{\phi(H')}] = \Ev_{\ldb H, H'\rdb},
  \end{equation*}
  where~$[\cdot\,, \cdot]$ is the anti-commutator. A bivector is a Poisson
  structure if $\ldb H, H\rdb = 0$, i.e., if~$\Ev_{\phi(H)}$ is a nilpotent
  vector field. Two Poisson structures are compatible if~$\ldb H, H'\rdb = 0$.
\end{definition}

\begin{remark}
  The bracket $\ldb\cdot\,, \cdot\rdb_H$ can be computed by
  \begin{equation}
    \label{eq:4}
    \ldb H, H'\rdb = \Ev_{\phi(H)}(H') + \Ev_{\phi(H')}(H).
  \end{equation}
\end{remark}

\begin{theorem}
  If~$H$ is a Poisson structure, then~$\{\cdot\,, \cdot\}_H$ is a Poisson
  bracket on the space of conservation laws, i.e., the bracket is
  skew-symmetric and enjoys the Jacobi identity.
\end{theorem}

\section{Examples}
\label{sec:examples}

We present here three examples of PDEs that admit bivectors which are not
Poisson structures. It is assumed everywhere that the jet order of the
operators are~$\leq$ the order of the equation.

\subsection{The wave equation}
\label{sec:equation-u_xy-=}

Consider the 2D wave equation in the form
\begin{equation}
  \label{eq:5}
  u_{xy} = 0.
\end{equation}
The cotangent covering is
\begin{equation*}
  \{p_{xy} = 0,\ u_{xy} = 0\}\to\{u_{xy} = 0\}.
\end{equation*}
Let
\begin{equation*}
  H = H_0 + H_1D_x + H_2D_y + H_3D_x^2 + H_4D_y^2,
\end{equation*}
where $H_i = H_i(x, y, u, u_x, u_y, u_{xx}, u_{yy})$. Then
\begin{gather*}
  \ell_{\mathcal{E}}\circ H = D_x\circ D_y\circ (H_0 + H_1D_x + H_2D_y +
  H_3D_x^2 + H_4D_y^2),\\
  H^*\circ \ell_{\mathcal{E}}^* = (H_0 - D_x\circ H_1 - D_y\circ H_2 +
  D_x^2\circ H_3 + D_y^2\circ H_4)\circ D_x\circ D_y.
\end{gather*}
Solving the system
\begin{equation*}
  \ell_{\mathcal{E}}\circ H = H^*\circ \ell_{\mathcal{E}}^*,\quad D_xD_y(H_u)
  = 0,
\end{equation*}
where $H_u = H(p) =  H_0p + H_1p_x + H_2p_y + H_3p_{xx} + H_4p_{yy}$, we get
the following solutions:
\begin{equation*}
  \label{eq:6}
  B_0 = p,\quad
  B_1 = \frac{1}{2}D_x(h_1) p_x + h_1 p_{xx},\quad
  B_2 = \frac{1}{2}D_y(h_2) p_y + h_2 p_{yy},
\end{equation*}
where $h_1= h_1(x, u_x)$, $h_2 = h_2(y, u_y)$.

\begin{remark}
  The bivector~$B_0$ is the identical Poisson structure, that exists
  on~$\mathcal{E}$, because~$\ell_{\mathcal{E}}$ is a self-adjoint
  operator. The bivectors~$B_1$ and~$B_2$ transform to each other by the
  discrete symmetry~$x\leftrightarrow y$ of~$\mathcal{E}$.
\end{remark}

\begin{proposition}
  The bivector~$B_1$ is a Poisson structure on Equation~\eqref{eq:5} if and
  only if~$\partial h_1/\partial u_x = 0$. Similarly\textup{,}~$B_2$ is a
  Poisson structure iff~$\partial h_2/\partial u_y = 0$. All these structures
  are pair-wise compatible.
\end{proposition}

\begin{proof}
  Let us first find the operator~$\nabla$ from
  Equation~\eqref{eq:8}. Extending all necessary objects to the
  space~$J^\infty(\hat{\pi})$, where
  $\hat{\pi}\colon \mathbb{R}^4(p, u, x, y) \to \mathbb{R}^2(x, y)$, by
  tedious but straightforward computations we find that
  \begin{equation*}
    \ell_F\circ H- H^*\circ\ell_F^* = D_x(F)A_1 + FA_0,
  \end{equation*}
  where
  \begin{equation*}    
    A_1 =
    \frac{\left(\dfrac{\partial h_1'}{\partial x} +
          \dfrac{\partial 
            h_1'}{\partial u_x} u_{xx}\right) p_x}{2} + h_1'
      p_{xx}\\
  \end{equation*}
    and
    \begin{align*}
      A_0
      &= 
        \left(\frac{\partial^2 h_1'}{\partial x\partial u_x} u_{xx}
        + \frac{\dfrac{\partial^2h_1'}{\partial x^2}
        + \dfrac{\partial^2h_1'}{\partial^2u_x^2}u_{xx}^2
        + \dfrac{\partial h_1'}{\partial u_x}u_{xxx}}{2}
        \right) p_x\\
      &+ \frac{3}{2}\left(
        \frac{\partial h_1'}{\partial x} + \frac{\partial h_1'}{\partial u_x}
        u_{xx} 
        \right) p_{xx} + h_1' p_{xxx}.
    \end{align*}
  Here $h_1' = \partial h_1/\partial u_x$. Consequently,
  \begin{equation*}
    \nabla^{*_1}(\cdot, p) =
    -\frac{1}{2}D_x(\cdot)\left(
      \left(\frac{\partial h_1'}{\partial x} + \frac{\partial h_1'}{\partial
          u_x}u_{xx}\right) p_x 
      + 2h_1'p_{xx}\right)
  \end{equation*}
  and
  \begin{equation*}
    \nabla^{*_1}(p, p) =
    -\frac{1}{2}p_x\left(
      \left(\frac{\partial h_1'}{\partial x} + \frac{\partial h_1'}{\partial
          u_x}u_{xx}\right) p_x 
      + 2h_1'p_{xx}\right) = -h_1'p_xp_{xx},
  \end{equation*}
  since $p_x\cdot p_x = 0$. Hence the components of the symmetry~$\phi(H_1)$
  are
  \begin{equation*}
    H_u = \frac{1}{2}D_x(h_1) p_x + h_1 p_{xx},\quad
    H_p = \frac{1}{2}\frac{\partial h_1}{\partial u_x}p_xp_{xx}.
  \end{equation*}
  Let us now compute $\ldb H_1, H_1\rdb$. One has by Equation~\eqref{eq:4}:
  \begin{multline*}
    \ldb B_1, B_1\rdb = 2\Ev_{\phi(H_1)}(H_u) \\= 2\left(D_x(H_p)\frac{\partial
        H_u}{\partial p_x} + D_x^2(H_p)\frac{\partial
        H_u}{\partial p_{xx}} + D_x(H_u)\frac{\partial
        H_u}{\partial u_x} + D_x^2(H_u)\frac{\partial
        H_u}{\partial u_{xx}}\right)\\
    = \frac{1}{2}\big(D_x(h_1')p_xp_{xx} + h_1'p_xp_{xxx}\big)D_x(h_1)\\
    + \big(D_x^2(h_1')p_xp_{xx} + 2D_x(h_1')p_xp_{xxx} + h_1'p_{xx}p_{xxx} +
    h_1'p_xp_{xxxx}\big)h_1\\
    + \big(D_x^2(h_1)p_x + 3D_x(h_1)p_{xx} +
    2h_1p_{xxx}\big)\big(\frac{1}{2}D_x(h_1')p_x + h_1'p_{xx}\big)\\
    + \frac{1}{2}\big(D_x^3(h_1)p_x + 4D_x^2(h_1)p_{xx} + 5D_x(h_1)p_{xxx} +
    2h_1p_{xxxx}\big)h_1'\\
    = -h_1h_1'p_{xx}p_{xxx} + (h_1D_x(h_1') - 2D_x(h_x)h_1')p_xp_{xxx}\\ +
    \big(h_1D_x^2(h_1') - D_x(h_1)D_x(h_1') - D_x^2(h_1)h_1'\big)p_xp_{xx},
  \end{multline*}
  which proves the result for~$H_1$. The proof for $H_2$ is the same; to check
  compatibility is a quite straightforward computation.
\end{proof}

We shall omit technical details in what follows.

\subsection{A third-order equation}
\label{sec:equation-u_xyz-=}

Consider the equation
\begin{equation}
  \label{eq:9}
  u_{xyz} = 0
\end{equation}
and find bivectors of the form
\begin{multline*}
  H = H_0p + H_1p_x + H_2p_y + H_3p_z + H_4p_{xx} + H_5p_{xy} + H_6p_{xz} +
  H_7p_{yy} + H_8p_{yz}\\
  + H_9p_{zz} + H_{10}p_{xxx} + H_{11}p_{xxy} + H_{12} p_{xxz} +
  H_{13} p_{xyy} + H_{14} p_{xzz}\\
  + H_{15} p_{yyy} + H_{16} p_{yyz} + H_{17} p_{yzz} + H_{18} p_{zzz}, 
\end{multline*}
where  $H_i$ depend on
\begin{multline*}
  x, y, z, u, u_x, u_y, u_z, u_{xx}, u_{xy}, u_{xz}, u_{yy}, u_{yz},
  u_{zz},\\ u_{xxx}, u_{xxy}, u_{xxz}, u_{xyy}, u_{xzz}, u_{yyy}, u_{yyz},
  u_{yzz}, u_{zzz}. 
\end{multline*}
By direct computations, the following solutions are found
\begin{gather*}
  B_1 = h p_x,\\
  B_2 = \frac{\partial h}{\partial x}p_{xx} + h p_{xxx},\\
  B_3 = \frac{1}{2}D_x(g) p_{xy} + g p_{xxy},\\
  B_4 = \frac{1}{2}\left(\frac{\partial g}{\partial x} u_{xyy} -
    \frac{\partial g}{\partial y} u_{xxy}\right) p_{xy} + g(u_{xyy} p_{xxy} -
  u_{xxy} p_{xyy}),
\end{gather*}
and another eight obtained from the above presented by the transpositions
$x\leftrightarrow y\leftrightarrow z\leftrightarrow x$. Here~$h = h(x)$, $g =
g(x, y, u_{xy})$.

\begin{proposition}
  The bivectors~$B_1$, $B_2$, and~$B_3$ are Poisson structures on
  Equation~\eqref{eq:9}. The bivector~$B_4$ is a Poisson structures if and
  only if~$g = 0$. All such structures are pair-wise compatible.
\end{proposition}

\begin{proof}
  Let us describe the operator~$\nabla$ for~$B_3$ and~$B_4$ first. One has
  for~$B_3$:
  \begin{equation*}
    \nabla_3(\cdot, p)= (\cdot)A_0(p) + D_x(\cdot)A_1(p) + D_y(\cdot)A_2(p) +
    D_xD_y(\cdot)A_{12}(p), 
  \end{equation*}
  where
  \begin{align*}
    A_0(p)
    &= \frac{1}{2}D_x^2D_y(g')p_{xy}
      + \frac{3}{2}D_xD_y(g') p_{xxy}+ \frac{1}{2}D_x^2(g')p_{xyy}\\
    &+ \frac{3}{2}D_x(g')p_{xxyy} + D_y(g') p_{xxxy}+ g' p_{xxxyy},\\
    A_1(p)
    &= \frac{1}{2}(D_xD_y(g')p_{xy} + D_x(g')p_{xyy} + D_y(g')p_{xxy})
      + g' p_{xxyy},\\
    A_2(p)
    &= \frac{1}{2}D_x^2(g')p_{xy} + \frac{3}{2}D_x(g')p_{xxy} + g'p_{xxxy},\\
    A_{12}(p)
    &= \frac{1}{2}D_x(g')p_{xy} + g'p_{xxy}.
  \end{align*}
  Here $g' = \partial g/\partial u_{xy}$. Consequently,
  \begin{equation*}
    \nabla_3^{1_*}(\cdot,p) = D_xD_y(\cdot)
    \left(\frac{1}{2}D_x(g')p_{xy} + g' p_{xxy}\right)
  \end{equation*}
  and
  \begin{equation}
    \label{eq:10}
    \nabla_3^{1_*}(p,p) = p_{xy}
    \left(\frac{1}{2}D_x(g')p_{xy} + g' p_{xxy}\right) = g'p_{xy} p_{xxy}.
  \end{equation}
  The expression for $\nabla_4(\cdot, p)$ is to big to present it here, but
  \begin{equation*}
    \nabla_4^{*_1}(\cdot, p) = D_xD_y(\cdot)
    \left(\frac{1}{2}\Big(u_{xyy}\frac{\partial g}{\partial x}
      - u_{xxy}\frac{\partial g'}{\partial y}\Big) p_{xy}
    + g'(u_{xyy} p_{xxy} - u_{xxy} p_{xyy})\right)
  \end{equation*}
  and
  \begin{equation}
    \label{eq:11}
    \nabla_4^{*_1}(p, p) = g'p_{xy}(u_{xyy} p_{xxy} - u_{xxy} p_{xyy})
  \end{equation}
  So,
  \begin{gather*}
    H_{3,u} = \frac{1}{2}D_x(g)p_{xy} +gp_{xxy},\quad H_{3,p} =
    \frac{1}{2}g'p_{xxy}p_{xy} ,\\
    H_{4,u} =  \frac{1}{2}\left(\frac{\partial g}{\partial x} u_{xyy} -
    \frac{\partial g}{\partial y} u_{xxy}\right) p_{xy} + g(u_{xyy} p_{xxy} -
  u_{xxy} p_{xyy}),\\
  H_{4,p} =
    \frac{1}{2}g'(u_{xyy} p_{xxy} - u_{xxy} p_{xyy})p_{xy}. 
  \end{gather*}
  Consequently,
  \begin{multline*}
    \ldb B_3, B_3\rdb = 2\Big(
    D_xD_y(H_{3,p})\frac{\partial H_{3,u}}{\partial p_{xy}} +
    D_x^2D_y(H_{3,p})\frac{\partial H_{3,u}}{\partial p_{xxy}}\\
    +
    D_xD_y(H_{3,u})\frac{\partial H_{3,u}}{\partial u_{xy}} +
    D_x^2D_y(H_{3,u})\frac{\partial H_{3,u}}{\partial u_{xxy}}
    \Big) = 2gg'p_{xxxxyy}p_{xy} + O(5),
  \end{multline*}
  where~$O(5)$ denotes the terms that depend on~$p_\sigma$ with~$\abs{\sigma}
  \leq 5$. The result follows from the fact that~$\ldb B_3, B_3\rdb$ is
  bilinear in~$g$ and~$g'$.

  In a similar way,
  \begin{multline*}
    \ldb B_4, B_4\rdb = 2\Big(
    D_xD_y(H_{4,p})\frac{\partial H_{4,u}}{\partial p_{xy}} +
    D_x^2D_y(H_{4,p})\frac{\partial H_{4,u}}{\partial p_{xxy}}\\
    +
    D_xD_y^2(H_{4,p})\frac{\partial H_{4,u}}{\partial p_{xyy}} +
    D_xD_y(H_{4,u})\frac{\partial H_{4,u}}{\partial u_{xy}}\\ +
    D_x^2D_y(H_{4,y})\frac{\partial H_{4,u}}{\partial u_{xxy}} +
    D_xD_y^2(H_{4,u})\frac{\partial H_{4,u}}{\partial u_{xyy}}
    \Big).
  \end{multline*}
  Notice now that in the expression for~$\ldb B_4, B_4\rdb$ there is a unique
  term of the form
  \begin{equation*}
    -2g^2(u_{xyy}p_{xxxxyy}p_{xyy} + u_{xxy}p_{xxyyyy}p_{xxy}),
  \end{equation*}
  and this completes the proof.
\end{proof}

\subsection{The Laplace equation}
\label{sec:laplace-equation}

Our last example is the 2D Laplace equation
\begin{equation}
  \label{eq:12}
  u_{xx} + u_{yy} = 0.
\end{equation}
Let
\begin{equation*}
  H = H_0p + H_1p_x + H_2p_y + H_3p_{xy} + H_4p_{yy},
\end{equation*}
where $H_i = H_i(x, y, u, u_x, u_y, u_{xy}, u_{yy})$. Then~$H$ is a bivector
on Equation~\eqref{eq:12} if~$H_0$ is a constant and
\begin{gather*}
  H_1 = \frac{1}{2}\left(\frac{\partial^2h}{\partial x\partial u_x}u_{yy}
    - \frac{\partial^2 h}{\partial x\partial u_y}u_{xy}
    - \frac{\partial^2h}{\partial x^2}\right),\\
  H_2 = \frac{1}{2}\left(\frac{\partial^2h}{\partial x\partial u_x}u_{xy}
    + \frac{\partial^2h}{\partial x\partial u_y}u_{yy}
    + \frac{\partial^2h}{\partial x\partial y}\right),\\
  H_3 = \frac{\partial h}{\partial y},\quad  H_4 = \frac{\partial h}{\partial
    x}, 
\end{gather*}
where the function~$h = h(x, y, u_x, u_y)$ enjoys the system
\begin{gather*}
  \frac{\partial^2h}{\partial x^2} + \frac{\partial^2h}{\partial y^2} = 0,
  \quad
  \frac{\partial^2h}{\partial x\partial u_x} - \frac{\partial^2h}{\partial
    y\partial u_y} = 0,\\
  \frac{\partial^2h}{\partial y\partial u_x} + \frac{\partial^2h}{\partial
    x\partial u_y} = 0,\quad
  \frac{\partial^3h}{\partial y\partial u_y^2} + \frac{\partial^3u}{\partial
    y\partial u_x^2} = 0.
\end{gather*}
We failed to find a generic solution and solved it under the assumption
that~$h$ is quadratic in its variables. Then the solutions are
\begin{align*}
  B_0
  &= p,\quad B_1 = p_{yy},\quad B_2 = p_{xy},\\
  B_3
  &= p_y + 2(xp_{xy} + yp_{yy}), \quad B_4 = p_x + 2(yp_{xy} - xp_{yy}),\\
  B_5
  &= u_{yy}p_y - u_{xy}p_x + 2(u_yp_{yy} - u_xp_{xy}),\\
  B_6
  &= u_{yy}p_x + u_{xy}p_y + 2(u_yp_{xy} + u_xp_{yy}),\\
  B_7
  &= (u_y + xu_{xy} + yu_{yy})p_x + (u_x + yu_{xy} - xu_{yy})p_y\\
  &+ 2(yu_x - xu_y)p_{yy} + 2(xu_x + yu_y)p_{xy},\\
  B_8 &= -(u_x +yu_{xy} - xu_{yy})p_x + (u_y + xu_{xy} + yu_{yy})p_y\\
  &+ 2(xu_x + yu_y)p_{yy} - 2(yu_x - xu_y)p_{xy}.
\end{align*}

\begin{proposition}
  The bivectors~$B_0$--$B_4$ are Poisson structures on the 2D Laplace
  equation, while~$B_5$--$B_8$ are not and the brackets $\ldb B_i,B_j\rdb$,
  $5\leq i\leq j\leq 8$ are nontrivial $3$-vectors on this equation.
\end{proposition}

The proof follows the same lines as the previous ones and we skip it.

\section{Concluding remarks}
\label{sec:concluding-remarks}

We constructed examples of PDEs that admit non-Poissonian bivectors. As a
consequence, this implies existence of nontrivial $3$-vectors. Thus, the
non-degeneracy of symbol in~\cite[Theorem~1]{integrability-conditions} is an
essential condition indeed. On the other hand, the 3D example
(Section~\ref{sec:equation-u_xyz-=}) seems to us somewhat degenerate; in other
3D examples: the Poincar\'{e} equation
(Appendix~\ref{sec:poinc-beginth-bibit}) and the 3D Laplace equation
(Appendix~\ref{sec:3d}), all the bivectors are Poissonian. This leads to a
conjecture that in the $n$-dimensional case, $n>2$, generically (whatever this
means) nontrivial $k$-vectors, $k>2$, do not exist.

\section*{Acknowledgements}
\label{sec:acknowledgements}

Computations were done with the \textsc{Jets} software,~\cite{BM-Jets}.

\appendix

\section{A fact}
\label{sec:fact}

The following statement is evident:
\begin{proposition}
  All bivectors of the form $B = \sum_\sigma B_\sigma p_\sigma$,
  where~$B_\sigma$ depend on~$x$, $y$, $z$ only, are pair-wise compatible
  Poisson structures.
\end{proposition}

\section{The 3D Laplace equation}
\label{sec:3d}

For the equation~$u_{xx} + u_{yy} + u_{zz} = 0$ the bivectors of jet order~$2$
are
\begin{align*}
  B_{0} &= p,\\
  B_1 &= p_{zz},\\
  B_2 &= p_{yy},\\
  B_3 &= p_{yz},\\
  B_4 &= p_{yz},\\
  B_5 &= p_{xz},\\
  B_6 &= xp_{xy} - xp_{yz},\\
  B_7 &= yp_{xz} - xp_{yz},\\
  B_8 &= 2zp_{xz} - 2xp_{zz} + p_x,\\
  B_9 &= 2zp_{yz} - 2yp_{zz} + p_y,\\
  B_{10} &= -2z p_{yy} + 2y p_{yz} + p_z,\\
  B_{11} &= xp_{xy} + yp_{yy} - zp_{yz} + 2y p_{zz},\\
  B_{12} &= xp_{xz} + 2 zp_{yy} - yp_{yz} + zp_{zz},\\
  B_{13} &= -z^2p_{yy}  + 2 y z p_{yz} - y^2p_{zz}  + yp_y + zp_z,\\
  B_{14} &= x p_x + z p_z + 2 x z p_{xz} + z^2 p_{yy} + (z^2-x^2) p_{zz},\\
  B_{15} &= 3yp_x + (y^2-x^2  - 3z^2) p_{xy} + 4 y z p_{xz} - 2 x y p_{yy} + 2
           x z p_{yz} - 4 x y p_{zz},\\ 
  B_{16} &= 4 x zp_{xy}  + 2 x yp_{xz}  + 4 y zp_{yy}  - 3x^2 p_{yz}  -y^2 p_{yz} 
           + z^2 p_{yz}\\
        &+ 2 y z p_{zz} + 3z p_y,\\
  B_{17} &= 3 z p_x + 4 y z p_{xy} + (z^2-x^2 - 3 y^2) p_{xz} - 4 x z p_{yy} + 2 x
           y p_{yz} - 2 x z p_{zz},\\ 
  B_{18} &= 2 x z p_{xy} + 4 x y p_{xz} + 2 y z p_{yy} - 3 x^2 p_{yz} + y^2
           p_{yz} - z^2 p_{yz}\\
        &+ 4 y z p_{zz} + 3y p_z,\\
  B_{19} &= 3 x p_z - 2 x z p_{yy} + 2 x z p_{zz} + (x^2 - 3 y^2 - z^2) p_{xz} +
           2 y z p_{xy} + 4 x y p_{yz},\\ 
  B_{20} &= 3 x p_y + 2 x y p_{yy} - 2 x y p_{zz} + (x^2 - y^2 - 3 z^2) p_{xy} +
        2 y z p_{xz} + 4 x z p_{yz},\\ 
  B_{21} &= 3 y z p_x - 3 x y p_z - 2 x y z p_{yy} - 4 x y z p_{zz} - z (x^2 -
           y^2 + z^2) p_{xy}\\
  &+ x (x^2 - y^2 + z^2) p_{yz} - 2 y (x^2 - z^2) p_{xz},\\
  B_{22} &= 3 x z p_y - 3 x y p_z + 2 x y z p_{yy} - 2 x y z p_{zz} + z (x^2 -
           y^2 - z^2) p_{xy}\\
  &- y (x^2 - y^2 - z^2) p_{xz} - 2 x (y^2 - z^2) p_{yz},\\
  B_{23} &= 3 x y p_x - y (3 x^2 - y^2 - 3 z^2) p_{yy} + 2 y^3 p_{zz} - x (x^2
           - 3 y^2 - 3 z^2) p_{xy}\\
  &- z (3 x^2 + 3 y^2 - z^2) p_{yz} + (-3 x^2 + 3 z^2) p_y - 3 yz p_z,\\
  B_{24} &= 3 z x p_x - z (3 x^2 - 3 y^2 - z^2) p_{zz} + 2 z^3 p_{yy} - x (x^2
           - 3 y^2 - 3 z^2) p_{xz}\\
  &- y (3 x^2 - y^2 + 3 z^2) p_{yz} + (-3 x^2 + 3 y^2) p_z - 3 zy p_y,\\
  B_{25} &= x p + 2 x y p_y - 2 x (x^2 + y^2 - 3 z^2) p_{zz} + 4 x z^2 p_{yy}
           + 2 z (3 x^2 + y^2 - z^2) p_{xz}\\
  & - 4 y z^2 p_{xy}+ 4 x y z p_{yz} + (3 x^2 + y^2 - 5 z^2) p_x + 8 x zp_z,\\
  B_{26} &= 3 y p + 2 y (3 x^2 - y^2 - 9 z^2) p_{yy} - 2 y (3 x^2 + 5 y^2 - 3
           z^2) p_{zz}\\
  &+ 2 x (x^2 - 3 y^2 - 9 z^2) p_{xy} + 4 z (3 x^2 + 6 y^2 - 2 z^2) p_{yz} +
    12 x y z p_{xz}\\
  &+ (9 x^2 + 9 y^2 - 21 z^2) p_y + 30 yz p_z,\\
  B_{27} &= 3 z p - 2 z (3 x^2 - 3 y^2 + 5 z^2) p_{yy} + 2 z (3 x^2 - 9 y^2 -
           z^2) p_{zz}\\
        &+ 2 x (x^2 - 9 y^2 - 3 z^2) p_{xz} + 4 y (3 x^2 - 2 y^2 + 6 z^2) p_{yz}\\
  &+ 12 x y z p_{xy} + (9 x^2 - 21 y^2 + 9 z^2) p_z + 30 zy p_y,\\
  B_{28} &= 3 x y p + 3 y (3 x^2 - y^2 - z^2) p_x - 3 x (x^2 - 3 y^2 + z^2) p_y +
        12 xyz p_z\\
  &+ (-x^4 + 6 x^2 y^2 - (y^2 + z^2) (y^2 - z^2)) p_{xy} + 2 y z (3 x^2 - y^2
    - z^2) p_{xz}\\
        &- 4 x y (x - y) (x + y) p_yy - 2 x z (x^2 - 3 y^2 + z^2) p_{yz}\\
  &- 2 x y
    (x^2 - y^2 - 3 z^2) p_{zz},\\ 
  B_{29} &= 3 z x p + 3 z (3 x^2 - y^2 - z^2) p_x + 12 y x z p_y - 3 x (x^2 +
           y^2 - 3 z^2) p_z\\
  &+ 2 y z (3 x^2 - y^2 - z^2) p_{xy} + (-x^4 + 6 x^2 z^2 + y^4 - z^4)
    p_{xz}\\
        &- 2 x z (x^2 - 3 y^2 - z^2) p_{yy} - 2 x y (x^2 + y^2 - 3 z^2)
          p_{yz}\\ 
      &- 4 x z (x - z) (x + z) p_{zz},\\
  B_{30} &= 12 x y z p_x - 3 z (x^2 - 3 y^2 + z^2) p_y - 3 y (x^2 + y^2 - 3
 z^2) p_z\\
  &- 2 x z (x^2 - 3 y^2 + z^2) p_{xy} + (-2 x y^3 + (-2 x^3 + 6 x z^2) y)
    p_{xz}\\
        &- 2 y z (3 x^2 - y^2 + z^2) p_{yy} + (x^4 - y^4 + 6 y^2 z^2 - z^4)
          p_{yz}\\ 
  &- 2 y z (3 x^2 + y^2 - z^2) p_{zz} + 3 y z p,\\
  B_{31} &= (x^2 + 2 x z + y^2 - z^2) (x^2 - 2 x z + y^2 - z^2) p_{zz}\\
  &- 4 (x -
    y) (x + y) z^2 p_{yy} - 4 x z (x^2 + y^2 - z^2) p_{xz}\\
  &- 4 y z (x^2 + y^2 -
    z^2) p_yz + 8 x y z^2 p_{xy} - 2 x (x^2 + y^2 - 5 z^2) p_x\\
  &- 2 y (x^2 + y^2 -
    5 z^2) p_y - 4 z (2 x^2 + 2 y^2 - z^2) p_z\\
  &- (x^2 + y^2 - 2 z^2) p.
\end{align*}

\section{The Poincar\'{e} equation}
\label{sec:poinc-beginth-bibit}

For the Poincar\'{e} equation $u_{xx} + u_{yy} - a^2u_{zz} = 0$, where~$a$ is
a nonzero constant, the result looks similar:
\begin{align*}
  B_0 &= p,\\
  B_1 &= p_{yy},\\
  B_2 &= p_{zz},\\
  B_3 &= p_{yz},\\
  B_4 &= p_{xy},\\
  B_5 &= p_{xz},\\
  B_6 &= yp_{xz} - xp_{yz},\\
  B_7 &= a^2x p_{yz} + zp_{xy},\\
  B_8 &=  a^2 p_z + 2 a^2y p_{yz} +2z p_{yy},\\
  B_9 &= p_x + 2 a^2x p_{zz} + 2z p_{xz},\\
  B_{10} &= p_y + 2 a^2y p_{zz} + 2z p_{yz},\\
  B_{11} &= -2 a^2y p_{zz} + xp_{xy} + yp_{yy} - zp_{yz},\\
  B_{12} &= a^2x p_{xz} - a^2y p_{yz} + a^2z p_{zz} - 2 zp_{yy},\\
  B_{13} &= -a^2x p_{zz} + yp_{xy} - zp_{xz} - xp_{yy},\\
  B_{14} &= 2 z p_z + 2 x z p_{xz} + (a^2(x^2 + 2 y^2) + z^2)p_{zz} + (x^2 -
           y^2) p_{yy}\\
  &- 2 x y p_{xy} + 2 y z p_{yz},\\
  B_{15} &= 3 y p_x a^2 + ((-x^2 + y^2) a^2 + 3 z^2) p_{xy} + 4 a^4 xy p_{zz}
           + 4a^2 y z p_{xz}\\
  &- 2a^2 x yp_{yy}  + 2a^2 x z p_{yz},\\  
  B_{16} &= 3 z p_x -4 x z p_{yy} + 2a^2 x zp_{zz}
           + ((x^2 + 3 y^2) a^2 + z^2) p_{xz}\\
      &- 2 a^2 y x p_{yz} + 4 y z p_{xy},\\
  B_{17} &= 3 a^2 x p_z - 2 y z p_{xy} + ((x^2 - 3 y^2) a^2 + z^2) p_{xz}
           + 4 a^2 y x p_{yz}\\
      &+ 2a^2 x z  p_{zz} + 2 x z p_{yy},\\
  B_{18} &= 3 z p_y + 3 a^2 y p_z + 2 x z p_{xy} + 2 a^2 y x p_{xz}
           + 2 y z p_{yy}\\
           &+ 2 (a^2 y^2 + z^2) p_{yz} + 2a^2 z  y p_{zz},\\
  B_{19} &= 3 a^2 x p_y + 2a^2y z p_{xz}  + 2a^2 x y p_{yy}
           + 2a^4 x  y p_{zz}\\
  &+ ((x^2 - y^2) a^2 + 3 z^2) p_{xy} + 4a^2 x z p_{yz} ,\\
  B_{20} &= 2 a^2 y p_y + 2 a^2 y x p_{xy} - 2a^2 x z p_{xz} + 2a^2 zyp_{yz}\\
      &+ ((-x^2 + y^2) a^2 + 2 z^2) p_{yy} - a^2 (a^2 x^2 + z^2) p_{zz},\\
  B_{21} &= 2 a^2 x p_x + 2 a^2 y x p_{xy} + 2a^2 x zp_{xz}
           + ((-x^2 + y^2) a^2 - 2 z^2) p_{yy}\\
  &- 2a^2z y p_{yz} + a^2 ((x^2 - 2 y^2) a^2 + z^2)p_{zz},\\
  B_{22} &= 3a^2x z p_y + 3a^4x y p_z + a^2 ((x^2 - y^2) a^2 + z^2) y p_{xz}
           + 2a^2x y z p_{yy}\\
      &+ 2a^4xz y p_zz + ((x^2 - y^2) a^2 + z^2)zp_{xy}
        + 2 a^2 x (a^2 y^2 + z^2) p_{yz},\\
  B_{23} &= (3 a^2 y^2 + 3 z^2) p_x + 3 a^2 x z p_z-3 x a^2 y p_y
           + x ((x^2 - 3 y^2) a^2 - 3 z^2) p_{yy}\\
  &+ a^2 x ((x^2 + 3 y^2) a^2 + 3 z^2) p_{zz} - y (3 a^2 x^2 - a^2 y^2 - 3
    z^2) p_{xy}\\
      &+ (3 z (x^2 + y^2) a^2 + z^3) p_{xz},\\
  B_{24} &= 3a^2z p  + 12a^2 x z p_x  + 18a^2y z p_y
           + 3 a^2 ((x^2 + 3 y^2) a^2 + 3 z^2) p_z\\
      &+ 12a^2 x y z p_{xy} + 2 a^2 x ((x^2 + 3 y^2) a^2 + 3 z^2) p_{xz} \\
  &- 2 z ((3 x^2 - 3 y^2) a^2 - z^2) p_{yy} + 4 a^2 y (a^2 y^2 + 3 z^2)
    p_{yz}\\
  &+ 2 z a^2 ((3 x^2 + 3 y^2) a^2 + z^2) p_{zz},\\
  B_{25} &= 3 a^2 y p + 12 x a^2 y p_x + 18 a^2 y z p_z - 2 y ((3 x^2 - y^2)
           a^2 - 3 z^2) p_{yy}\\
  &+ 2 a^2 y ((3 x^2 + y^2) a^2 + 3 z^2) p_{zz} - 2 x ((x^2 - 3 y^2) a^2 - 3
    z^2) p_{xy}\\
      &+ 4 z (3 a^2 y^2 + z^2) p_{yz} + 12a^2x y z p_{xz}
        + ((-3 x^2 + 9 y^2) a^2 + 9 z^2) p_y,\\
  B_{26} &= 3 a^2 x p + (9 a^2 x^2 + 12 z^2) p_x + 9a^2x y p_y + 21 a^2 x z
      p_z\\
  &- x ((x^2 - 3 y^2) a^2 + 9 z^2) p{_yy} + a^2 x ((5 x^2 + 3 y^2) a^2 + 15
    z^2) p_{zz}\\
  &+ y ((3 x^2 - y^2) a^2 + 9 z^2) p_{xy} + z ((15 x^2 + 3 y^2) a^2 + 5 z^2)
    p_{xz}\\
  &+ 12a^2x y z p_{yz} ,\\
B_{27} &= 9 a^2 y z p_x + (9 x + 3) y a^4 p_z - z ((3 x^2 - 3 y^2 + 2 x) a^2 -
         3 z^2) p_{xy}\\
      &+ 2 y a^2 ((3 x^2 + 2 x) a^2 + 3 z^2) p_{xz} - 2 (3 x + 1) y z a^2 p_{yy}\\
  &- a^2 ((3 x^3 - 3 x y^2 + 3 x^2 - y^2) a^2 + (-3 x - 1) z^2) p_{yz} + 4 (3 x + 1) z y a^4 p_{zz},\\
      B_{28} &= 3 a^2yzp + 12 a^2x y zp_x - 3 z ((x^2 - 3 y^2) a^2 - z^2)
               p_y\\
  &+ 3 a^2 y ((x^2 + y^2) a^2 + 3 z^2) p_z
    - 2 x z ((x^2 - 3 y^2) a^2 - z^2) p_{xy}\\
  &+ 2 a^2 x ((x^2 + y^2) a^2 + 3 z^2) y p_{xz} - 2 z y ((3 x^2 - y^2) a^2 -
    z^2) p_{yy}\\
  &+ ((-x^4 + y^4) a^4 + 6 a^2 y^2 z^2 + z^4) p_{yz} + 2 a^2 y z ((3 x^2 +
    y^2) a^2 + z^2) p_{zz},\\ 
  B_{29} &= 3 a^2 (a^2 y^2 + z^2) p + 12 a^2 (a^2 y^2 + z^2) x p_x - 6 a^2
           ((x^2 - y^2) a^2 - 3 z^2) y p_y\\
  &+ 6 a^2 z ((x^2 + 3 y^2) a^2 + z^2) p_z - 4 a^2 x ((x^2 - y^2) a^2 - 3 z^2)
    y p_{xy}\\
      &+ 4 a^2 x z ((x^2 + 3 y^2) a^2 + z^2) p_{xz}\\
      &+ ((x^4 - 6 x^2 y^2 + y^4) a^4 + 6(-x^2 z^2 + y^2 z^2) a^2 + z^4) p_{yy}\\
      &+ 8 a^2 z y (a^2 y^2 + z^2) p_{yz}\\
  &+ a^2 ((x^4 + 6 x^2 y^2 + y^4) a^4 + 6 z^2 (x^2 + y^2) a^2 + z^4) p_{zz},\\
  B_{30} &= 9 a^2 x z p + 3 z ((9 x^2 - 3 y^2 + 2 x) a^2 + 3 z^2) p_x + 6 z
           a^2 y (6 x - 1) p_y \\
      &+ 2 a^2 z ((6 x^3 + 3 x^2 - 3 y^2) a^2 + (6 x + 1) z^2) p_{zz}\\
      &- 2 z ((3 x^3 - 9 x y^2) a^2 + (3 x + 2) z^2) p_{yy}\\
  &+ ((3 x^4 - 3 y^4 + 2 x^3 - 6 x y^2) a^4 + (18 x^2 z^2 + 6 x z^2) a^2 + 3
    z^4) p_{xz}\\
      &+ 2 a^2 ((3 x^3 + 3 x y^2 + 3 x^2 - y^2) a^2 + (9 x - 3) z^2) y p_{yz}\\
  &+ 6 z y ((3 x^2 - y^2) a^2 + z^2) p_{xy} + 3 a^2 ((3 x^3 + 3 x y^2 + 2 x^2 - 2 y^2) a^2 + 9 x z^2) p_z,\\
  B_{31} &= 9 a^4 x y p + 3 a^2 ((9 x^2 - 3 y^2 + 2 x) a^2 + 3 z^2) y p_x\\
  &- 3 a^2 ((3 x^3 - 9 x y^2 + 2 x^2) a^2 + (-3 x + 2) z^2) p_y + 6 a^4 (6 x -
    1) z y p_z\\
  &+ ((-3 y^4 + (18 x^2 + 6 x) y^2 - 3 x^4 - 2 x^3) a^4 - 6 x z^2 a^2 + 3 z^4)
    p_{xy}\\
  &+ 6 a^2 z ((3 x^2 - y^2) a^2 + z^2) y p_{xz} - 2 a^2 ((6 x^3 - 6 x y^2 + 3
    x^2 - y^2) a^2 + 3 z^2) y p_{yy}\\
      &+ 2 a^4 ((3 x^3 - 3 x y^2 - 2 y^2) a^2 + 9 x z^2) y p_{zz}\\
  &- 2 a^2 ((3 x^3 - 9 x y^2 + 3 x^2 + 3 y^2) a^2 + (-3 x + 1) z^2) z p_{yz}.
\end{align*}
\end{document}